\documentclass[11pt,reqno]{amsart}

\usepackage[T1]{fontenc}
\usepackage{lmodern}
\usepackage{amsmath,amssymb,amsthm,mathtools,mathrsfs,bm}
\usepackage{xcolor}
\usepackage[protrusion=true,expansion=false]{microtype}
\usepackage{enumitem}
\usepackage[colorlinks=true,linkcolor=blue!45!black,citecolor=blue!45!black,urlcolor=blue!45!black]{hyperref}
\usepackage[nameinlink,capitalise,noabbrev]{cleveref}

\allowdisplaybreaks
\numberwithin{equation}{section}
\setlist[itemize]{leftmargin=2em,itemsep=2pt,topsep=4pt}
\setlist[enumerate]{leftmargin=2.4em,itemsep=2pt,topsep=4pt}

\newtheorem{theorem}{Theorem}[section]
\newtheorem{proposition}[theorem]{Proposition}
\newtheorem{lemma}[theorem]{Lemma}
\newtheorem{corollary}[theorem]{Corollary}

\theoremstyle{definition}

\newtheorem{remark}[theorem]{Remark}

\newcommand{\N}{\mathbb N}
\newcommand{\R}{\mathbb R}
\newcommand{\C}{\mathbb C}
\newcommand{\K}{\mathbb K}
\newcommand{\E}{\mathbb E}
\newcommand{\Pp}{\mathbb P}
\newcommand{\Om}{\Omega_n}
\newcommand{\muN}{\mu_n}
\newcommand{\eps}{\varepsilon}
\newcommand{\1}{\mathbf 1}
\newcommand{\wh}{\widehat}
\newcommand{\cX}{\mathfrak X}
\newcommand{\cM}{\mathfrak M}
\newcommand{\betaP}{\beta_p}
\newcommand{\ellTwo}{\ell_2^n}
\newcommand{\norm}[1]{\left\lVert #1\right\rVert}
\newcommand{\abs}[1]{\left\lvert #1\right\rvert}
\newcommand{\ip}[2]{\left\langle #1,#2\right\rangle}
\newcommand{\asympU}{\asymp}

\title[Sharp metric $X_p$ inequalities]{Sharp $p/\log p$ Bounds for Metric $X_p$ Inequalities and Arbitrary Rademacher Chaos}

\author{Yutong Zhang}
\address{School of Mathematics, Sichuan University,
24 First Loop Road South Section I,
Chengdu 610064, Sichuan, China}
\email{yutongzhang@stu.scu.edu.cn}
\thanks{Corresponding author: Yutong Zhang}

\author{Yaoran Yang}
\address{School of Mathematics, Sichuan University,
24 First Loop Road South Section I,
Chengdu 610064, Sichuan, China}
\email{yangyaoran@stu.scu.edu.cn}


\begin{document}

\begin{abstract}
Naor proved that $L_p$ satisfies the metric $X_p$ inequality with the sharp scaling parameter and asked for the optimal dependence of its constant on $p$. The best previously known upper bound was $O(p^4/\log p)$, while first-degree chaos gives a lower bound of order $p/\log p$. We prove the matching upper bound. More generally, if $h$ is a mean-zero scalar function on the Hamming cube, $p\geq2$, and $S$ is uniformly distributed over the $k$-subsets of $[n]$, then
\[
 \bigl(\E_S\|\E_{[n]\setminus S}h\|_p^p\bigr)^{1/p}
 \lesssim \frac p{\log p}
 \left[\frac{k}{n}\sum_{j=1}^n\|\partial_jh\|_p^p+
 \left(\frac{k}{n}\right)^{p/2}\|h\|_p^p\right]^{1/p}.
\]
The proof combines a sharp fixed-cardinality Rosenthal comparison with the exact reconstruction $h=\sum_jD_j\Delta^{-1}h$. A dimension-free inverse-gradient square-function estimate follows by integrating a pointwise reverse Poincar\'e inequality for the Walsh heat semigroup. Naor's transference argument then yields the optimal metric bound $O(p/\log p)$.
\end{abstract}

\maketitle

\section{Introduction and main results}\label{sec:introduction}

For $p\geq2$, the metric $X_p$ inequality of Naor and Schechtman is a nonlinear counterpart of a sharp moment inequality for randomly restricted Rademacher sums. Naor proved that $L_p$ satisfies the metric inequality with the optimal scaling parameter, but the resulting constant was of order $p^4/\log p$; first-degree chaos forces a lower bound of order $p/\log p$ \cite{NaorSchechtman2016,Naor2016}. We prove that Walsh chaos of arbitrary degree entails no additional loss in the dependence on $p$.

Write $[n]=\{1,\ldots,n\}$ and $\Om=\{-1,1\}^n$, and let $\muN$ be the uniform probability measure on $\Om$. All expectations over finite sets are normalized. For $S\subseteq[n]$, let $P_S=\E_{[n]\setminus S}$ denote conditional expectation over the coordinates outside $S$. If $\eps^{(j)}$ is obtained from $\eps$ by changing the sign of its $j$th coordinate, set $D_jh(\eps)=\bigl(h(\eps)-h(\eps^{(j)})\bigr)/2$ and $\partial_jh=2D_jh$. We write $L_p^0(\Om;\K)=\{h\in L_p(\Om;\K):\E h=0\}$, put $\betaP=p/\log p$, and use natural logarithms throughout.

Our first result is the arbitrary-chaos estimate.

\begin{theorem}[Sharp $X_p$ inequality for arbitrary Walsh chaos]\label{thm:main-chaos}
There is a universal constant $C\in(0,\infty)$ such that, for every $p\in[2,\infty)$, $n\in\N$, $k\in[n]$, $\K\in\{\R,\C\}$, and $h\in L_p^0(\Om;\K)$,
\begin{align}
 \left(\binom nk^{-1}\sum_{\substack{S\subseteq[n]\\|S|=k}}
 \norm{P_Sh}_{L_p(\Om)}^p\right)^{1/p}
 &\leq C\betaP\Bigg[
 \frac{k}{n}\sum_{j=1}^n\norm{\partial_jh}_{L_p(\Om)}^p\notag\\
 &\hspace{31mm}+\left(\frac{k}{n}\right)^{p/2}
 \norm{h}_{L_p(\Om)}^p\Bigg]^{1/p}.
\label{eq:main-chaos}
\end{align}
The order $p/\log p$ is optimal up to universal multiplicative constants.
\end{theorem}

The argument gives the following more precise estimate.

\begin{theorem}[Refined arbitrary-chaos estimate]\label{thm:refined-chaos}
There is a universal constant $C\in(0,\infty)$ with the following property. Let $p\in[2,\infty)$, $n\in\N$, $k\in[n]$, $t=k/n$, and let $S$ be uniform among the $k$-subsets of $[n]$. Then every $h\in L_p^0(\Om;\K)$, where $\K\in\{\R,\C\}$, satisfies
\begin{equation}
 \left(\E_S\norm{P_Sh}_p^p\right)^{1/p}
 \leq t\norm{h}_p+C\betaP
 \left(t\sum_{j=1}^n\norm{D_jh}_p^p\right)^{1/p}
 +\frac{C\pi}{2}\betaP\sqrt t\,\norm{h}_p.
\label{eq:refined-chaos}
\end{equation}
Moreover, if $(M,\nu)$ is a measure space and $h:\Om\to L_p(M,\nu)$ satisfies $\E_\eps h(\eps)=0$ in $L_p(M,\nu)$, then the same estimate holds with the corresponding Bochner norms, over either scalar field.
\end{theorem}

For the metric statement, let $G=\mathbb Z_{8m}^n$, where $\mathbb Z_{8m}=\mathbb Z/(8m\mathbb Z)$, let $e_1,\ldots,e_n$ be its coordinate vectors, and write $\eps_S=\sum_{j\in S}\eps_je_j$.

\begin{theorem}[Sharp metric $X_p$ inequality for $L_p$]\label{thm:main-metric}
There is a universal constant $C\in(0,\infty)$ with the following property. Let $p\in[2,\infty)$, $n\in\N$, $k\in[n]$, and $m\in\N$ satisfy $m\geq\sqrt{n/k}$. For every measure space $(M,\nu)$ and every mapping $f:\mathbb Z_{8m}^n\to L_p(M,\nu)$, over either the real or the complex scalar field,
\begin{align}
&\left(
 \binom nk^{-1}\sum_{\substack{S\subseteq[n]\\|S|=k}}
 \E_{x,\eps}\norm{f(x+4m\eps_S)-f(x)}_{L_p(M)}^p
 \right)^{1/p}\notag\\
&\quad\leq C\frac p{\log p}\,m\Bigg[
 \frac{k}{n}\sum_{j=1}^n\E_x\norm{f(x+e_j)-f(x)}_{L_p(M)}^p\notag\\
&\hspace{49mm}+\left(\frac{k}{n}\right)^{p/2}
 \E_{x,\eps}\norm{f(x+\eps)-f(x)}_{L_p(M)}^p
 \Bigg]^{1/p}.
\label{eq:main-metric}
\end{align}
The order $p/\log p$ cannot be improved up to universal factors.
\end{theorem}

We now define the optimal constants without reference to unspecified constants in the preceding statements. For fixed $p\geq2$, let $\cX_p$ be the least $K\in[0,\infty]$ such that, for every $n\in\N$, $k\in[n]$, $\K\in\{\R,\C\}$, and $h\in L_p^0(\Om;\K)$,
\begin{equation}
 \left(\binom nk^{-1}\sum_{\substack{S\subseteq[n]\\|S|=k}}\norm{P_Sh}_p^p\right)^{1/p}
 \leq K\left[
 \frac{k}{n}\sum_{j=1}^n\norm{\partial_jh}_p^p+
 \left(\frac{k}{n}\right)^{p/2}\norm h_p^p\right]^{1/p}.
\label{eq:def-X-constant}
\end{equation}
Likewise, let $\cM_p$ be the least $K\in[0,\infty]$ such that, for every $n\in\N$, $k\in[n]$, $m\in\N$ with $m\geq\sqrt{n/k}$, every measure space $(M,\nu)$, either choice of scalar field, and every mapping $f:\mathbb Z_{8m}^n\to L_p(M,\nu)$,
\begin{align}
&\left(
 \binom nk^{-1}\sum_{\substack{S\subseteq[n]\\|S|=k}}
 \E_{x,\eps}\norm{f(x+4m\eps_S)-f(x)}_{L_p(M)}^p
 \right)^{1/p}\notag\\
&\quad\leq Km\Bigg[
 \frac{k}{n}\sum_{j=1}^n\E_x\norm{f(x+e_j)-f(x)}_{L_p(M)}^p\notag\\
&\hspace{46mm}+\left(\frac{k}{n}\right)^{p/2}
 \E_{x,\eps}\norm{f(x+\eps)-f(x)}_{L_p(M)}^p
 \Bigg]^{1/p}.
\label{eq:def-M-constant}
\end{align}

\begin{theorem}\label{thm:constant-asymptotics}
There are universal constants $0<c\leq C<\infty$ such that, for every $p\geq2$,
\[
 c\frac p{\log p}\leq\cX_p\leq C\frac p{\log p},
 \qquad
 c\frac p{\log p}\leq\cM_p\leq C\frac p{\log p}.
\]
\end{theorem}

The analytic mechanism is the exact decomposition $g_j=D_j\Delta^{-1}h$. It simultaneously satisfies $\sum_jg_j=h$, the restriction factorization $P_Sh=P_S(\sum_{j\in S}g_j)$, the coordinate estimate $\|g_j\|_p\leq\|D_jh\|_p$, and the square-function estimate
\begin{equation}
 \norm{\left(\sum_{j=1}^n\abs{D_j\Delta^{-1}h}^2\right)^{1/2}}_{L_p(\Om)}
 \leq\frac\pi2\norm{h}_{L_p(\Om)}.
\label{eq:square-introduction}
\end{equation}
The last inequality is obtained by integrating a pointwise reverse Poincar\'e estimate for the Walsh heat semigroup. The heat-kernel coefficient identity used in that estimate appears in \cite{IvanisviliVanHandelVolberg2020}; the reverse Poincar\'e formulation is also recorded explicitly in \cite{IvanisviliXieZhang2026}. The additional ingredient in the present argument is therefore not the Bessel step in isolation, but its combination with the exact restriction factorization and the fixed-cardinality Rosenthal comparison.

The paper is organized as follows. \Cref{sec:inverse-gradient} develops the Walsh calculus, the exact reconstruction, and the inverse-gradient estimate. \Cref{sec:sampling-chaos} proves the fixed-cardinality sampling lemma and the arbitrary-chaos estimates. \Cref{sec:metric-transfer} carries out the metric transference. \Cref{sec:optimality} establishes optimality and records two consequences. \Cref{sec:context} compares the argument with the factorization through Riesz transforms and places the result in the existing literature.

\section{Walsh calculus and inverse-gradient estimates}\label{sec:inverse-gradient}

\subsection{Fourier multipliers and the heat semigroup}

For standard background on Walsh analysis on the discrete cube, see \cite{ODonnell2014}. For $A\subseteq[n]$, let $W_A(\eps)=\prod_{j\in A}\eps_j$, with $W_\varnothing\equiv1$. Every scalar function $h:\Om\to\C$ has the Walsh expansion $h=\sum_{A\subseteq[n]}\wh h(A)W_A$, where $\wh h(A)=\E_\eps[h(\eps)W_A(\eps)]$. The coordinate derivatives and the number operator $\Delta=\sum_{j=1}^nD_j$ act diagonally: $D_jW_A=\1_{\{j\in A\}}W_A$ and $\Delta W_A=|A|W_A$. For a mean-zero function, define $\Delta^{-1}h=\sum_{\varnothing\neq A\subseteq[n]}|A|^{-1}\wh h(A)W_A$. The projection $P_S$ is given by $P_Sh=\sum_{A\subseteq S}\wh h(A)W_A$ and is contractive on every $L_q(\Om)$, $1\leq q\leq\infty$.

For $0\leq\rho\leq1$, the Walsh noise semigroup is $T_\rho h=\sum_A\rho^{|A|}\wh h(A)W_A$. Equivalently, if $\xi_1,\ldots,\xi_n$ are independent signs satisfying $\Pp(\xi_j=1)=(1+\rho)/2$ and $\E\xi_j=\rho$, then $T_\rho h(\eps)=\E_\xi h(\eps\xi)$. Thus $T_\rho$ is a positive contraction on every $L_q(\Om)$. We also use the coordinate-deleted noise operator $T_\rho^{(j^c)}$, defined by $T_\rho^{(j^c)}W_A=\rho^{|A\setminus\{j\}|}W_A$; it is again an $L_q$ contraction.

If $\E h=0$, then $\rho^{-1}T_\rho h$, initially defined for $\rho>0$, extends continuously to $\rho=0$, because the Walsh expansion of $h$ has no constant term. Consequently,
\begin{equation}
 \Delta^{-1}h=\int_0^1T_\rho h\,\frac{d\rho}{\rho}.
\label{eq:delta-inverse-integral}
\end{equation}
Indeed, the coefficient of every nonempty $W_A$ on the right is $\int_0^1\rho^{|A|-1}\,d\rho=1/|A|$. Since $D_j$ commutes with $T_\rho$ and every Walsh frequency of $D_jh$ contains $j$,
\begin{equation}
 D_j\Delta^{-1}h
 =\int_0^1T_\rho D_jh\,\frac{d\rho}{\rho}
 =\int_0^1T_\rho^{(j^c)}D_jh\,d\rho.
\label{eq:coordinate-integral}
\end{equation}
Both integrals are finite-dimensional Bochner integrals, with the integrands understood through their continuous extensions at $\rho=0$.

\begin{lemma}[Coordinate contraction]\label{lem:coordinate-contraction}
For every $1\leq p\leq\infty$, every scalar $h:\Om\to\C$, and every $j\in[n]$, one has $\|D_j\Delta^{-1}(h-\E h)\|_p\leq\|D_jh\|_p$.
\end{lemma}

\begin{proof}
Since $D_j(h-\E h)=D_jh$, \eqref{eq:coordinate-integral}, the triangle inequality for Bochner integrals, and the contractivity of $T_\rho^{(j^c)}$ give
\[
 \norm{D_j\Delta^{-1}(h-\E h)}_p
 \leq\int_0^1\norm{T_\rho^{(j^c)}D_jh}_p\,d\rho
 \leq\norm{D_jh}_p.
\]
This argument also covers $p=\infty$.
\end{proof}

\subsection{Exact reconstruction and restriction factorization}

Let $h$ be mean zero and set $g_j=D_j\Delta^{-1}h=\sum_{A\ni j}|A|^{-1}\wh h(A)W_A$ for $j\in[n]$.

\begin{proposition}[Reconstruction and restriction factorization]\label{prop:factorization}
For every mean-zero scalar function $h$ and every $S\subseteq[n]$,
\begin{equation}
 \sum_{j=1}^ng_j=h,
 \qquad
 P_Sh=P_S\left(\sum_{j\in S}g_j\right).
\label{eq:factorization}
\end{equation}
Consequently, $\|P_Sh\|_p\leq\|\sum_{j\in S}g_j\|_p$ for every $1\leq p\leq\infty$.
\end{proposition}

\begin{proof}
For every nonempty $A\subseteq[n]$, the coefficient of $W_A$ in $\sum_jg_j$ is $|A|^{-1}\sum_{j=1}^n\1_{\{j\in A\}}\wh h(A)=\wh h(A)$. Since $\wh h(\varnothing)=0$, this proves the first identity. Moreover, $\sum_{j\in S}g_j=\sum_{\varnothing\neq A\subseteq[n]}\frac{|A\cap S|}{|A|}\wh h(A)W_A$. Applying $P_S$ removes every term with $A\nsubseteq S$. For each surviving nonempty $A\subseteq S$, the multiplier $|A\cap S|/|A|$ equals $1$, so the result is $P_Sh$. The norm estimate follows from the $L_p$ contractivity of $P_S$.
\end{proof}

\subsection{Reverse Poincar\'e and the inverse-gradient square function}

Fix $0<\rho<1$ and let $\xi_1,\ldots,\xi_n$ have the biased-sign law used above. Define $\zeta_j=(\xi_j-\rho)/\sqrt{1-\rho^2}$. Independence gives $\E\zeta_j=0$, $\E\zeta_j^2=1$, and $\E(\zeta_i\zeta_j)=0$ for $i\neq j$; hence $1,\zeta_1,\ldots,\zeta_n$ is an orthonormal family in the biased product $L_2$ space.

\begin{lemma}[Biased coefficient identity]\label{lem:biased-coefficient}
For every $h:\Om\to\C$, $j\in[n]$, $\eps\in\Om$, and $0<\rho<1$,
\begin{equation}
 T_\rho D_jh(\eps)
 =\frac{\rho}{\sqrt{1-\rho^2}}\,
 \E_\xi\bigl[\zeta_jh(\eps\xi)\bigr].
\label{eq:biased-coefficient}
\end{equation}
Both sides extend continuously to zero at $\rho=0$.
\end{lemma}

\begin{proof}
By linearity, it suffices to take $h=W_A$. If $j\notin A$, both sides vanish. If $j\in A$, then $\E(\zeta_j\xi_j)=\sqrt{1-\rho^2}$, and independence gives
\[
 \E_\xi[\zeta_jW_A(\eps\xi)]
 =W_A(\eps)\sqrt{1-\rho^2}\,\rho^{|A|-1}.
\]
Multiplication by $\rho/\sqrt{1-\rho^2}$ yields $\rho^{|A|}W_A(\eps)=T_\rho D_jW_A(\eps)$.
\end{proof}

\begin{proposition}[Pointwise reverse Poincar\'e estimate]\label{prop:reverse-poincare}
For every scalar $h:\Om\to\C$, every $\eps\in\Om$, and every $0<\rho<1$,
\begin{align}
 \sum_{j=1}^n\abs{T_\rho D_jh(\eps)}^2
 &\leq\frac{\rho^2}{1-\rho^2}
 \left(T_\rho(|h|^2)(\eps)-\abs{T_\rho h(\eps)}^2\right)\notag\\
 &\leq\frac{\rho^2}{1-\rho^2}T_\rho(|h|^2)(\eps).
\label{eq:reverse-poincare}
\end{align}
Consequently, for every $p\geq2$,
\begin{equation}
 \norm{\left(\sum_{j=1}^n\abs{T_\rho D_jh}^2\right)^{1/2}}_{L_p(\Om)}
 \leq\frac{\rho}{\sqrt{1-\rho^2}}\norm{h}_{L_p(\Om)}.
\label{eq:reverse-poincare-Lp}
\end{equation}
\end{proposition}

\begin{proof}
Fix $\eps$ and put $H_\eps(\xi)=h(\eps\xi)$. By \Cref{lem:biased-coefficient},
\[
 (T_\rho D_jh(\eps))_{j=1}^n
 =\frac{\rho}{\sqrt{1-\rho^2}}
 (\ip{H_\eps}{\zeta_j}_{L_2(\xi)})_{j=1}^n.
\]
Bessel's inequality for the orthonormal family $1,\zeta_1,\ldots,\zeta_n$ gives
\[
 \abs{\E_\xi H_\eps}^2+
 \sum_{j=1}^n\abs{\E_\xi[H_\eps\zeta_j]}^2
 \leq\E_\xi\abs{H_\eps}^2.
\]
Since $\E_\xi H_\eps=T_\rho h(\eps)$ and $\E_\xi|H_\eps|^2=T_\rho(|h|^2)(\eps)$, multiplication by $\rho^2/(1-\rho^2)$ proves the first inequality in \eqref{eq:reverse-poincare}; the second is immediate.

Taking square roots and then the $L_p(\eps)$ norm gives $\|(\sum_j|T_\rho D_jh|^2)^{1/2}\|_p\leq\frac{\rho}{\sqrt{1-\rho^2}}\|(T_\rho|h|^2)^{1/2}\|_p$. Since $p/2\geq1$ and $T_\rho$ is an $L_{p/2}$ contraction, the last norm is at most $\|h\|_p$.
\end{proof}

\begin{theorem}[Dimension-free inverse-gradient square function]\label{thm:inverse-gradient-square}
If $p\in[2,\infty)$, $\K\in\{\R,\C\}$, and $h\in L_p^0(\Om;\K)$, then
\begin{equation}
 \norm{\left(\sum_{j=1}^n\abs{D_j\Delta^{-1}h}^2\right)^{1/2}}_{L_p(\Om)}
 \leq\frac\pi2\norm{h}_{L_p(\Om)}.
\label{eq:inverse-gradient-square}
\end{equation}
The constant is independent of $n$ and $p$; no assertion is made that $\pi/2$ is optimal.
\end{theorem}

\begin{proof}
For $0<r<1$, define $\bm g^{(r)}=\int_0^r(T_\rho D_1h,\ldots,T_\rho D_nh)\,\frac{d\rho}{\rho}\in L_p(\Om;\ellTwo)$. The integrand extends continuously to $\rho=0$. Minkowski's integral inequality and \eqref{eq:reverse-poincare-Lp} give
\begin{align*}
 \norm{\bm g^{(r)}}_{L_p(\Om;\ellTwo)}
 &\leq\int_0^r
 \norm{(T_\rho D_jh)_{j=1}^n}_{L_p(\Om;\ellTwo)}\,\frac{d\rho}{\rho}\\
 &\leq\int_0^r\frac{d\rho}{\sqrt{1-\rho^2}}\norm{h}_p
 =\arcsin(r)\norm{h}_p.
\end{align*}
By the finite Walsh expansion and \eqref{eq:coordinate-integral}, $\bm g^{(r)}$ converges in $L_p(\Om;\ellTwo)$ as $r\uparrow1$ to $(D_j\Delta^{-1}h)_{j=1}^n$. Passing to the limit and using $\arcsin(r)\uparrow\pi/2$ proves the theorem.
\end{proof}

\section{Fixed-cardinality sampling and the arbitrary-chaos estimate}\label{sec:sampling-chaos}

The probabilistic input is the sharp-order Rosenthal inequality. There is a universal constant $C_R$ such that, whenever $p\geq2$ and $X_1,\ldots,X_N$ are independent mean-zero random variables with values in $\K\in\{\R,\C\}$,
\begin{equation}
 \norm{\sum_{r=1}^NX_r}_{L_p}
 \leq C_R\frac p{\log p}\left[
 \left(\sum_{r=1}^N\E|X_r|^p\right)^{1/p}
 +\left(\sum_{r=1}^N\E|X_r|^2\right)^{1/2}
 \right].
\label{eq:sharp-rosenthal}
\end{equation}
For real random variables this is the classical sharp-order Rosenthal inequality; the complex case follows by applying the real inequality to real and imaginary parts. The order $p/\log p$ is optimal \cite{Rosenthal1970,JSZ1985,Pinelis2015}. When a sharp-order formulation is stated only for sufficiently large $p$, the remaining bounded range follows from the classical Rosenthal inequality; since $p/\log p$ is bounded below on $[2,p_0]$, enlarging the universal constant yields \eqref{eq:sharp-rosenthal} for all $p\geq2$.

\begin{lemma}[Fixed-cardinality sampling with sharp $p$-dependence]\label{lem:fixed-cardinality}
There is a universal constant $C_0$ such that the following holds. Let $p\geq2$, $n\in\N$, $k\in[n]$, $t=k/n$, and $a_1,\ldots,a_n\in\K$, where $\K\in\{\R,\C\}$. If $S$ is uniform among the $k$-subsets of $[n]$, then
\begin{align}
 \left(\E_S\abs{\sum_{j\in S}a_j}^p\right)^{1/p}
 &\leq t\abs{\sum_{j=1}^na_j}\notag\\
 &\quad+C_0\betaP\left[
 \left(t\sum_{j=1}^n|a_j|^p\right)^{1/p}
 +\sqrt t\left(\sum_{j=1}^n|a_j|^2\right)^{1/2}
 \right].
\label{eq:fixed-cardinality}
\end{align}
\end{lemma}

\begin{proof}
Let $I_1,\ldots,I_k$ be a uniformly random ordered sample of $k$ distinct indices from $[n]$. Then $\sum_{j\in S}a_j$ and $\sum_{r=1}^ka_{I_r}$ have the same distribution. Let $J_1,\ldots,J_k$ be independent and uniform on $[n]$, put $Y_r=a_{J_r}$, and set $a_{\mathrm{av}}=n^{-1}\sum_{j=1}^na_j$.

Assume first that the scalars are real. Hoeffding's convex-order comparison for sampling without replacement \cite[Theorem~4]{Hoeffding1963} says that every convex function of the without-replacement sum has expectation no larger than the corresponding function of the with-replacement sum; see also \cite{BardenetMaillard2015}. Taking $\Phi(u)=|u|^p$ gives
\begin{equation}
 \left(\E_S\abs{\sum_{j\in S}a_j}^p\right)^{1/p}
 \leq\norm{\sum_{r=1}^kY_r}_{L_p}.
\label{eq:hoeffding-lp}
\end{equation}
For complex scalars, put
\[
 c_p=\frac1{2\pi}\int_0^{2\pi}|\cos\theta|^p\,d\theta,
 \qquad
 |z|^p=\frac1{2\pi c_p}\int_0^{2\pi}
 |\operatorname{Re}(e^{i\theta}z)|^p\,d\theta.
\]
Apply the real comparison to the populations $(\operatorname{Re}(e^{i\theta}a_j))_{j=1}^n$ and integrate in $\theta$. This proves \eqref{eq:hoeffding-lp} over $\C$ as well.

Decompose $\sum_{r=1}^kY_r=ka_{\mathrm{av}}+\sum_{r=1}^k(Y_r-a_{\mathrm{av}})$. Since $ka_{\mathrm{av}}=t\sum_ja_j$, Minkowski's inequality and \eqref{eq:sharp-rosenthal} imply
\begin{align}
 \norm{\sum_{r=1}^kY_r}_{L_p}
 &\leq t\abs{\sum_{j=1}^na_j}
 +C_R\betaP\left[
 \left(k\E|Y_1-a_{\mathrm{av}}|^p\right)^{1/p}
 +\left(k\E|Y_1-a_{\mathrm{av}}|^2\right)^{1/2}
 \right].
\label{eq:rosenthal-centered}
\end{align}
On the uniform probability space $[n]$, Minkowski's and H\"older's inequalities give $(k\E|Y_1-a_{\mathrm{av}}|^p)^{1/p}\leq(\frac{k}{n}\sum_{j=1}^n|a_j|^p)^{1/p}+k^{1/p}|a_{\mathrm{av}}|\leq2(t\sum_{j=1}^n|a_j|^p)^{1/p}$. Moreover, $k\E|Y_1-a_{\mathrm{av}}|^2=k(\frac1n\sum_{j=1}^n|a_j|^2-|a_{\mathrm{av}}|^2)\leq t\sum_{j=1}^n|a_j|^2$.
Substitution in \eqref{eq:rosenthal-centered}, followed by \eqref{eq:hoeffding-lp}, proves \eqref{eq:fixed-cardinality} after enlarging the universal constant.
\end{proof}

\begin{remark}\label{rem:sampling-sharpness}
The adjective ``sharp'' in \Cref{lem:fixed-cardinality} refers only to the order $p/\log p$. No sharpness is asserted in the finite-population parameter $t=k/n$. For example, if $k=n$ and $\sum_ja_j=0$, then the left-hand side vanishes identically, whereas the two fluctuation terms on the right need not vanish. This loss is inherent in the comparison with sampling with replacement and is irrelevant for the application below.
\end{remark}

\begin{proof}[Proof of \Cref{thm:refined-chaos} for scalar functions]
Let $g_j=D_j\Delta^{-1}h$. By \Cref{prop:factorization} and the contractivity of $P_S$,
\begin{equation}
 \left(\E_S\norm{P_Sh}_p^p\right)^{1/p}
 \leq\left(\E_{\eps,S}\abs{\sum_{j\in S}g_j(\eps)}^p\right)^{1/p}.
\label{eq:chaos-reduction}
\end{equation}
For each fixed $\eps$, apply \Cref{lem:fixed-cardinality} to $a_j=g_j(\eps)$. The reconstruction identity $\sum_jg_j=h$ gives
\begin{align*}
 \left(\E_S\abs{\sum_{j\in S}g_j(\eps)}^p\right)^{1/p}
 &\leq t|h(\eps)|
 +C_0\betaP\left(t\sum_{j=1}^n|g_j(\eps)|^p\right)^{1/p}\\
 &\quad+C_0\betaP\sqrt t
 \left(\sum_{j=1}^n|g_j(\eps)|^2\right)^{1/2}.
\end{align*}
Taking the $L_p(\eps)$ norm and applying Minkowski's inequality yields
\begin{align}
 \left(\E_{\eps,S}\abs{\sum_{j\in S}g_j(\eps)}^p\right)^{1/p}
 &\leq t\norm{h}_p
 +C_0\betaP\left(t\sum_{j=1}^n\norm{g_j}_p^p\right)^{1/p}\notag\\
 &\quad+C_0\betaP\sqrt t\,
 \norm{\left(\sum_{j=1}^n|g_j|^2\right)^{1/2}}_p.
\label{eq:chaos-three-terms}
\end{align}
The first mixed norm is exact, and the square-function term is controlled by \Cref{thm:inverse-gradient-square}:
\begin{align*}
 \norm{\left(t\sum_j|g_j|^p\right)^{1/p}}_p^p
 &=t\sum_j\norm{g_j}_p^p,\\
 \norm{\left(\sum_j|g_j|^2\right)^{1/2}}_p
 &\leq\frac\pi2\norm h_p.
\end{align*}
Moreover, \Cref{lem:coordinate-contraction} gives $\|g_j\|_p\leq\|D_jh\|_p$ for every $j$. Substitution in \eqref{eq:chaos-three-terms}, followed by \eqref{eq:chaos-reduction}, proves \eqref{eq:refined-chaos}.
\end{proof}

\begin{proof}[Bochner-valued extension in \Cref{thm:refined-chaos}]
Let $h:\Om\to L_p(M,\nu)$ satisfy $\E_\eps h(\eps)=0$ in $L_p(M,\nu)$. Because $\Om$ is finite, choose representatives $h(\eps,z)$ for all $\eps$ on a common conull subset of $M$. Then $\E_\eps h(\eps,z)=0$ for almost every $z$.

Apply the scalar argument pointwise in $z$. The operators $D_j$, $P_S$, $T_\rho$, and $\Delta^{-1}$ act only in the $\eps$ variable, so the reconstruction and restriction identities hold for almost every $z$. The fixed-cardinality estimate is applied pointwise in $(\eps,z)$, and the reverse Poincar\'e estimate is applied to the scalar function $\eps\mapsto h(\eps,z)$. Tonelli's theorem and the canonical isometry $L_p(\Om;L_p(M,\nu))\cong L_p(\Om\times M,\muN\otimes\nu)$ then turn every scalar $L_p(\eps,z)$ norm into the corresponding Bochner norm. No constant changes, and the argument is valid over both scalar fields.
\end{proof}

\begin{proof}[Proof of \Cref{thm:main-chaos}]
Since $0<t\leq1$, one has $t\leq\sqrt t$; moreover, $\betaP\geq1$ for $p\geq2$. Hence \eqref{eq:refined-chaos} implies
\[
 \left(\E_S\norm{P_Sh}_p^p\right)^{1/p}
 \leq C_1\betaP\left[
 \left(t\sum_{j=1}^n\norm{D_jh}_p^p\right)^{1/p}
 +\sqrt t\,\norm{h}_p\right]
\]
with a universal $C_1$. Since $\partial_j=2D_j$ and $A+B\leq2^{1-1/p}(A^p+B^p)^{1/p}$ for $A,B\geq0$, this gives \eqref{eq:main-chaos} after changing the universal constant. The lower bound is proved in \Cref{sec:optimality}.
\end{proof}

\section{Deduction of the metric inequality}\label{sec:metric-transfer}

We now transfer \Cref{thm:main-chaos} to the metric statement. The argument follows the smoothing scheme of \cite{Naor2016}; every step is included in order to track the dependence on $p$. Fix $m,n\in\N$, put $G=\mathbb Z_{8m}^n$, and let $\K\in\{\R,\C\}$. All additions in this section take place in $G$, and all $L_p(G)$ norms use normalized counting measure.

For $A\subseteq[n]$ and $f:G\to\K$, define $T_Af(x)=\E_\delta f(x+2\delta_A)$, where $\delta_A=\sum_{j\in A}\delta_je_j$. This is an average of translations and hence an $L_p(G)$ contraction.

\begin{lemma}[Smoothing estimate]\label{lem:smoothing}
For every $1\leq p<\infty$, $A\subseteq[n]$, and $f:G\to\K$,
\begin{equation}
 \norm{f-T_Af}_{L_p(G)}
 \leq2\norm{f(x+\eps)-f(x)}_{L_p(x,\eps)}.
\label{eq:smoothing}
\end{equation}
\end{lemma}

\begin{proof}
Jensen's inequality gives $|f(x)-T_Af(x)|^p\leq\E_\delta|f(x)-f(x+2\delta_A)|^p$. Insert $x+\delta_A+\delta_{A^c}$ between the two endpoints and use $|u+v|^p\leq2^{p-1}(|u|^p+|v|^p)$. After averaging in $(x,\delta)$, the first increment is a uniform full-sign increment. By translation invariance in $x$, the second has the same average as the increment by $-\delta_A+\delta_{A^c}$, which is also uniformly distributed on $\{-1,1\}^n$. Therefore $\E_x|f(x)-T_Af(x)|^p\leq2^p\E_{x,\eps}|f(x+\eps)-f(x)|^p$. Taking $p$th roots proves the claim.
\end{proof}

For each $x\in G$, define the odd cube function $h_x(\eps)=f(x+2\eps)-f(x-2\eps)$. Since $h_x(-\eps)=-h_x(\eps)$, one has $\E_\eps h_x(\eps)=0$.

\begin{lemma}[Exact cube--torus identity]\label{lem:cube-torus}
For every $S\subseteq[n]$, $x\in G$, and $\eps\in\Om$,
\begin{equation}
 T_{S^c}f(x+2\eps_S)-T_{S^c}f(x-2\eps_S)=P_Sh_x(\eps).
\label{eq:cube-torus}
\end{equation}
\end{lemma}

\begin{proof}
Expanding the conditional expectation gives $P_Sh_x(\eps)=\E_{\eta_{S^c}}[f(x+2\eps_S+2\eta_{S^c})-f(x-2\eps_S-2\eta_{S^c})]$. The first average is $T_{S^c}f(x+2\eps_S)$. Since $-\eta_{S^c}$ has the same distribution as $\eta_{S^c}$, the second is $T_{S^c}f(x-2\eps_S)$.
\end{proof}

\begin{lemma}[Telescoping the smoothed long jump]\label{lem:smoothed-telescope}
For every $S\subseteq[n]$,
\begin{equation}
 \frac1m\norm{T_{S^c}f(x+4m\eps_S)-T_{S^c}f(x)}_{L_p(x,\eps)}
 \leq\norm{P_Sh_x(\eps)}_{L_p(x,\eps)}.
\label{eq:smoothed-telescope}
\end{equation}
\end{lemma}

\begin{proof}
For fixed $(x,\eps)$, telescope the long increment into $m$ increments of length $4\eps_S$:
\[
 T_{S^c}f(x+4m\eps_S)-T_{S^c}f(x)
 =\sum_{r=1}^m\bigl[T_{S^c}f(x+4r\eps_S)-T_{S^c}f(x+4(r-1)\eps_S)\bigr].
\]
Minkowski's inequality bounds the $L_p(x,\eps)$ norm by the sum of the norms of these terms. In the $r$th term, set $y=x+2(2r-1)\eps_S$. Translation invariance in $x$ shows that its norm equals $\|T_{S^c}f(y+2\eps_S)-T_{S^c}f(y-2\eps_S)\|_{L_p(y,\eps)}$, which is $\|P_Sh_y\|_{L_p(y,\eps)}$ by \Cref{lem:cube-torus}. All $m$ terms have the same norm, and division by $m$ proves the result.
\end{proof}

\begin{lemma}[Derivative and norm comparison]\label{lem:hx-comparison}
For every $p\geq1$, every $f:G\to\K$, and every $j\in[n]$,
\begin{align}
 \E_x\norm{\partial_jh_x}_{L_p(\Om)}^p
 &\leq8^p\E_x|f(x+e_j)-f(x)|^p,
\label{eq:hx-derivative}\\
 \E_x\norm{h_x}_{L_p(\Om)}^p
 &\leq4^p\E_{x,\eps}|f(x+\eps)-f(x)|^p.
\label{eq:hx-norm}
\end{align}
\end{lemma}

\begin{proof}
Since $\eps^{(j)}=\eps-2\eps_je_j$,
\begin{align*}
 \partial_jh_x(\eps)
 &=f(x+2\eps)-f(x+2\eps-4\eps_je_j)\\
 &\quad-f(x-2\eps)+f(x-2\eps+4\eps_je_j).
\end{align*}
Using $|u+v|^p\leq2^{p-1}(|u|^p+|v|^p)$ and averaging over $(x,\eps)$, translation invariance and the symmetry $\eps_j\leftrightarrow-\eps_j$ give $\E_x\|\partial_jh_x\|_p^p\leq2^p\E_y|f(y+4e_j)-f(y)|^p$. The four-step telescoping identity $f(y+4e_j)-f(y)=\sum_{r=1}^4[f(y+re_j)-f(y+(r-1)e_j)]$ and convexity yield $|f(y+4e_j)-f(y)|^p\leq4^{p-1}\sum_{r=1}^4|f(y+re_j)-f(y+(r-1)e_j)|^p$. After averaging in $y$, all four terms have the same mean, so the last mean is at most $4^p\E_y|f(y+e_j)-f(y)|^p$. This proves \eqref{eq:hx-derivative} because $2^p4^p=8^p$.

For \eqref{eq:hx-norm}, telescope from $x-2\eps$ to $x+2\eps$ in four full-sign steps: $h_x(\eps)=\sum_{r=-1}^{2}[f(x+r\eps)-f(x+(r-1)\eps)]$. Convexity bounds $|h_x(\eps)|^p$ by $4^{p-1}$ times the sum of the four $p$th powers. Averaging over $(x,\eps)$ and translating $x$ in each term turns every summand into $\E_{x,\eps}|f(x+\eps)-f(x)|^p$, proving \eqref{eq:hx-norm}.
\end{proof}

\begin{proposition}[Smoothed metric inequality]\label{prop:smoothed-metric}
Let $p\geq2$, $k\in[n]$, $t=k/n$, and $m\in\N$. Every scalar $f:G\to\K$ satisfies
\begin{align}
&\left(\E_{S,x,\eps}
 \abs{T_{S^c}f(x+4m\eps_S)-T_{S^c}f(x)}^p\right)^{1/p}\notag\\
&\quad\leq C_2\betaP m
 \left[t\sum_{j=1}^n\E_x|f(x+e_j)-f(x)|^p
 +t^{p/2}\E_{x,\eps}|f(x+\eps)-f(x)|^p
 \right]^{1/p},
\label{eq:smoothed-metric}
\end{align}
where $C_2$ is universal.
\end{proposition}

\begin{proof}
Taking the $L_p$ average over $S$ in \eqref{eq:smoothed-telescope} gives
\[
 \frac1m\left(\E_{S,x,\eps}
 |T_{S^c}f(x+4m\eps_S)-T_{S^c}f(x)|^p\right)^{1/p}
 \leq\left(\E_{S,x}\norm{P_Sh_x}_p^p\right)^{1/p}.
\]
For each fixed $x$, the function $h_x$ is mean zero. Apply \Cref{thm:main-chaos}, raise the resulting inequality to the $p$th power, and average in $x$ to obtain
\[
 \left(\E_{S,x}\norm{P_Sh_x}_p^p\right)^{1/p}
 \leq C\betaP\left[
 t\sum_{j=1}^n\E_x\norm{\partial_jh_x}_p^p
 +t^{p/2}\E_x\norm{h_x}_p^p\right]^{1/p}.
\]
By \Cref{lem:hx-comparison}, the quantity in brackets is at most
\[
 8^p\left[t\sum_{j=1}^n\E_x|f(x+e_j)-f(x)|^p
 +t^{p/2}\E_{x,\eps}|f(x+\eps)-f(x)|^p\right].
\]
Combining these estimates and absorbing the factor $8$ into the universal constant proves \eqref{eq:smoothed-metric}.
\end{proof}

\begin{proof}[Proof of \Cref{thm:main-metric} for scalar-valued functions]
Set
\[
 A(f)=t\sum_{j=1}^n\E_x|f(x+e_j)-f(x)|^p,
 \qquad
 B(f)=\E_{x,\eps}|f(x+\eps)-f(x)|^p.
\]
For each fixed $S$, insert the smoothed function at both endpoints and apply Minkowski's inequality:
\begin{align*}
 \norm{f(x+4m\eps_S)-f(x)}_{L_p(x,\eps)}
 &\leq\norm{T_{S^c}f(x+4m\eps_S)-T_{S^c}f(x)}_{L_p(x,\eps)}\\
 &\quad+\norm{f(x+4m\eps_S)-T_{S^c}f(x+4m\eps_S)}_{L_p(x,\eps)}\\
 &\quad+\norm{f(x)-T_{S^c}f(x)}_{L_p(x,\eps)}.
\end{align*}
The last two terms are equal by translation invariance in $x$, and \Cref{lem:smoothing} bounds each by $2B(f)^{1/p}$. Taking the $L_p$ average in $S$ and applying \Cref{prop:smoothed-metric} yields
\begin{equation}
 \left(\E_{S,x,\eps}|f(x+4m\eps_S)-f(x)|^p\right)^{1/p}
 \leq C_2\betaP m[A(f)+t^{p/2}B(f)]^{1/p}+4B(f)^{1/p}.
\label{eq:metric-error}
\end{equation}
The scaling assumption $m\geq t^{-1/2}$ is equivalent to $m\sqrt t\geq1$. Since $\betaP\geq1$, we have $4B(f)^{1/p}\leq4\betaP m\sqrt t\,B(f)^{1/p}=4\betaP m[t^{p/2}B(f)]^{1/p}$. This term is absorbed into the right-hand side of \eqref{eq:metric-error}, proving \eqref{eq:main-metric} for scalar functions. The scalar chaos estimate has already been proved over both fields, and the remaining transference argument uses only absolute values, Jensen's inequality, Minkowski's inequality, and translation invariance. Hence the conclusion holds over both $\R$ and $\C$.
\end{proof}

\begin{proof}[Completion of the proof of \Cref{thm:main-metric}]
Let $f:G\to L_p(M,\nu)$. Because $G$ is finite, choose representatives $F(x,z)$ for all $x\in G$ on a common conull subset of $M$. For almost every $z$, apply the scalar metric inequality to $x\mapsto F(x,z)$ and raise it to the $p$th power. Integration in $z$ and Tonelli's theorem turn the left-hand side into $\E_{S,x,\eps}\|f(x+4m\eps_S)-f(x)\|_{L_p(M)}^p$ and the two terms on the right into precisely the two Bochner-valued quantities in \eqref{eq:main-metric}. Taking $p$th roots completes the proof.
\end{proof}

\begin{remark}\label{rem:scaling}
The condition $m\geq\sqrt{n/k}$ is not used in the chaos estimate or in the smoothed long-jump estimate. It enters only when the unsmoothing error in \eqref{eq:metric-error} is absorbed by the full-sign increment term.
\end{remark}

\section{Optimality and consequences}\label{sec:optimality}

\begin{proposition}[Degree one forces $p/\log p$]\label{prop:degree-one-lower}
There is a universal $c>0$ such that $\cX_p\geq cp/\log p$ for every $p\geq2$.
\end{proposition}

\begin{proof}
Let $h(\eps)=\sum_{j=1}^na_j\eps_j$ with real coefficients. Then $\E h=0$, $P_Sh=\sum_{j\in S}a_j\eps_j$, and $\|\partial_jh\|_p=2|a_j|$. Hence \eqref{eq:def-X-constant} specializes to
\begin{equation}
 \left(\E_{S,\eps}\abs{\sum_{j\in S}a_j\eps_j}^p\right)^{1/p}
 \leq\cX_p\left[
 2^pt\sum_{j=1}^n|a_j|^p
 +t^{p/2}\E_\eps\abs{\sum_{j=1}^na_j\eps_j}^p
 \right]^{1/p}.
\label{eq:degree-one-specialization}
\end{equation}
Let
\begin{align*}
 R_0&=\left[t\sum_{j=1}^n|a_j|^p
 +t^{p/2}\E_\eps\abs{\sum_{j=1}^na_j\eps_j}^p\right]^{1/p},\\
 R_1&=\left[2^pt\sum_{j=1}^n|a_j|^p
 +t^{p/2}\E_\eps\abs{\sum_{j=1}^na_j\eps_j}^p\right]^{1/p}.
\end{align*}
Then $R_0\leq R_1\leq2R_0$. The sharp Johnson--Schechtman--Zinn inequality, in the fixed-cardinality form recorded in \cite[Eq.~(9)]{Naor2016}, asserts that the least constant in the inequality with $R_0$ on the right has order $p/\log p$ \cite{JMST1979,JSZ1985}. Since $R_1\leq2R_0$, the least constant with $R_1$ on the right is at least one half of the least constant with $R_0$ on the right. This proves $\cX_p\gtrsim p/\log p$ in the asymptotic range covered by the sharpness assertion. To make the uniform statement for all $p\geq2$ explicit, take $n=k=1$ and $h(\eps)=\eps$. Then \eqref{eq:def-X-constant} gives $1\leq\cX_p(2^p+1)^{1/p}$, whence $\cX_p\geq(2^p+1)^{-1/p}\geq1/\sqrt5$. Since $p/\log p$ is bounded on every bounded subinterval of $[2,\infty)$, decreasing the universal constant completes the proof for all $p\geq2$.
\end{proof}

\begin{proposition}[The metric lower bound]\label{prop:metric-lower}
There is a universal $c>0$ such that $\cM_p\geq cp/\log p$ for every $p\geq2$.
\end{proposition}

\begin{proof}
The inequality defining $\cM_p$ in \eqref{eq:def-M-constant} has the normalization of \cite[Eq.~(4)]{Naor2016}: the domain is $\mathbb Z_{8m}^n$, the long jump is $4m\eps_S$, the scaling assumption is $m\geq\sqrt{n/k}$, and the two terms on the right have the same powers of $k/n$. Naor's \cite[Remark~5]{Naor2016}, referring to the linear obstruction from \cite{NaorSchechtman2016}, yields the lower bound $\cM_p\gtrsim p/\log p$ in the asymptotic range. For completeness, a uniform positive lower bound in the remaining bounded range follows from a one-dimensional example. Take $n=k=m=1$, let $M$ be a one-point space, and define $f(x)=e^{\pi i x/4}$ on $\mathbb Z_8$. The long increment has modulus $2$, while both the unit increment and the full-sign increment have the constant modulus $a=|e^{\pi i/4}-1|$. Thus \eqref{eq:def-M-constant} implies $2\leq\cM_p2^{1/p}a$, and consequently $\cM_p\geq2^{1-1/p}/a$. Since $p/\log p$ is bounded on bounded intervals, a reduction of the universal constant proves the claim for every $p\geq2$.
\end{proof}

\begin{proof}[Proof of \Cref{thm:constant-asymptotics}]
The upper bounds follow from \Cref{thm:main-chaos,thm:main-metric}; the lower bounds are \Cref{prop:degree-one-lower,prop:metric-lower}.
\end{proof}

\begin{corollary}[No higher-order loss]\label{cor:no-higher-order-loss}
Let $\cX_p^{(d)}$ be defined as in \eqref{eq:def-X-constant}, but with $h$ restricted to Walsh polynomials of degree at most $d$. For every $d\in\N\cup\{\infty\}$ with $d\geq1$,
\[
 \cX_p^{(d)}\asympU\frac p{\log p},
 \qquad p\geq2,
\]
with universal comparison constants independent of $d$.
\end{corollary}

\begin{proof}
The upper bound follows from \Cref{thm:main-chaos} and is independent of the Walsh degree. The class of polynomials of degree at most $d$ contains every degree-one function when $d\geq1$, so \Cref{prop:degree-one-lower} gives the lower bound.
\end{proof}

\begin{corollary}[Separated statistics]\label{cor:separated-statistics}
Let $p\geq2$, $n\in\N$, $k\in[n]$, $t=k/n$, and let $S$ be uniform among the $k$-subsets of $[n]$. Then every $h\in L_p^0(\Om;\K)$, where $\K\in\{\R,\C\}$, satisfies
\[
 \left(\E_S\norm{P_Sh}_p^p\right)^{1/p}
 \lesssim t\norm{h}_p
 +\frac p{\log p}\left(t\sum_{j=1}^n\norm{D_jh}_p^p\right)^{1/p}
 +\frac p{\log p}\sqrt t\,\norm{h}_p.
\]
\end{corollary}

\begin{proof}
This is \Cref{thm:refined-chaos}, with universal constants suppressed.
\end{proof}

\begin{remark}[Scope of the vector-valued statement]\label{rem:vector-scope}
The extension to $L_p(M)$ uses the canonical isometric identification $L_p(\Om;L_p(M))\cong L_p(\Om\times M)$. The argument does not automatically extend to an arbitrary Banach target. In particular, the pointwise Bessel estimate in \Cref{prop:reverse-poincare} extracts an $\ell_2$ family of scalar coefficients before the outer $L_p$ norm is taken, and there is no general Banach-valued substitute compatible with the required Bochner structure.
\end{remark}

\section{Further remarks and comparison with previous work}\label{sec:context}

\subsection{The inverse-gradient estimate}

With our normalization, set $R_j=D_j\Delta^{-1/2}$. Dimension-free estimates for these discrete Riesz transforms were developed by Lust-Piquard and in subsequent work \cite{LustPiquard1998,LustPiquard2004,BenEfraimLustPiquard2008,JungeMeiParcet2018}. The operator needed here is $D_j\Delta^{-1}$, which contains an additional inverse half-power. Estimate \eqref{eq:inverse-gradient-square} follows directly by integrating the reverse Poincar\'e inequality, without first estimating the Riesz vector.

The coefficient identity \eqref{eq:biased-coefficient} is a scalar specialization of the representation in \cite[Lemma~2.1]{IvanisviliVanHandelVolberg2020}; related biased-sign representations also occur in \cite{ChenDai2026}. The pointwise estimate \eqref{eq:reverse-poincare}, including its interpretation as a reverse Poincar\'e inequality, is explicitly used in \cite{IvanisviliXieZhang2026}. Thus neither the coefficient identity nor the Bessel extraction is claimed as new. The argument here uses that standard estimate in the inverse-gradient decomposition required by the fixed-cardinality restriction problem.

\subsection{Losses from separate Riesz estimates}

Write $\mathbf Ru=(R_ju)_{j=1}^n$. The factorization $(D_j\Delta^{-1}h)_{j=1}^n=\mathbf R\Delta^{-1/2}h$ can be combined with the separate estimates $\|\mathbf Ru\|_{L_p(\ell_2)}\lesssim p\|u\|_p$ and $\|\Delta^{-1/2}\|_{L_p^0\to L_p^0}\lesssim\sqrt{\log p}$, valid for $p\geq2$; see \cite{DomelevoIvanisviliPetermichlVolberg2026} and \cite[Lemma~10]{Naor2016}. This direct composition gives only
\[
 \norm{(D_j\Delta^{-1}h)_{j=1}^n}_{L_p(\ell_2)}
 \lesssim p\sqrt{\log p}\,\norm{h}_p.
\]
Consequently, the square-function term in \Cref{lem:fixed-cardinality} would be bounded only by a universal multiple of $p^2(\log p)^{-1/2}\sqrt t\,\|h\|_p$, rather than by a quantity of order $(p/\log p)\sqrt t\,\|h\|_p$. This calculation shows that the naive procedure of estimating the two factors separately is quantitatively insufficient; it does not rule out a different argument exploiting cancellation in their composition. The direct reverse-Poincar\'e integration avoids both losses.

A coordinatewise estimate is also insufficient. The term produced by Minkowski is $\sqrt t\,\|(g_j)_j\|_{L_p(\ell_2)}\leq \sqrt t(\sum_j\|g_j\|_p^2)^{1/2}$, where $g_j=D_j\Delta^{-1}h$, and the right-hand side cannot be controlled dimension-freely by the derivative and norm statistics in \eqref{eq:refined-chaos}. We give an explicit example with $p=4$. Let $r\geq2$ be even, set $N=2^r$ and $n=N+r$, and identify $\Omega_n$ with $\Omega_N\times\Omega_r$. Enumerate $\Omega_r$ as $y^{(1)},\ldots,y^{(N)}$, let $J(y^{(i)})=i$, and define $h(x,y)=x_{J(y)}$. Then $\E h=0$ and $\|h\|_4=1$. For the $x_i$-coordinates,
\[
 D_{x_i}h=x_i\1_{\{y=y^{(i)}\}},
 \qquad \sum_{i=1}^N\norm{D_{x_i}h}_4^4=1.
\]
Flipping a $y$-coordinate replaces the selected $x$-coordinate by a distinct one. If $y^{(\ell)}$ denotes the vector obtained from $y$ by changing the sign of its $\ell$th coordinate, then, for every $\ell\in[r]$, $D_{y_\ell}h=(x_{J(y)}-x_{J(y^{(\ell)})})/2$, so $\|D_{y_\ell}h\|_4^4=1/2$. Consequently,
\[
 \sum_{j=1}^{n}\norm{D_jh}_4^4=1+\frac r2.
\]
To estimate the coordinatewise term, use the Walsh identity $\1_{\{y=y^{(i)}\}}=N^{-1}\sum_{B\subseteq[r]}W_B(y^{(i)})W_B(y)$. It gives
\[
 g_i(x,y)=D_{x_i}\Delta^{-1}h(x,y)
 =\frac{x_i}{N}\sum_{B\subseteq[r]}
 \frac{W_B(y^{(i)})W_B(y)}{1+|B|}.
\]
At $y=y^{(i)}$,
\[
 |g_i(x,y^{(i)})|
 =\frac1N\sum_{s=0}^r\binom rs\frac1{s+1}
 =\frac{2-N^{-1}}{r+1},
\]
where the last identity follows by integrating $(1+u)^r$ over $u\in[0,1]$. Since the fiber $\{y=y^{(i)}\}$ has normalized measure $N^{-1}$ and $|x_i|=1$, it follows that $\|g_i\|_4\geq N^{-1/4}(2-N^{-1})/(r+1)$. Consequently,
\[
 \left(\sum_{i=1}^N\norm{g_i}_4^2\right)^{1/2}
 \geq \frac{2-N^{-1}}{r+1}N^{1/4}.
\]
Because $r$ is even, $n$ is even; choose $k=n/2$, so $t=1/2$. The coordinatewise bound is then at least a universal multiple of $2^{r/4}/r$, whereas
\[
 \left(t\sum_{j=1}^{n}\norm{D_jh}_4^4\right)^{1/4}
 +\sqrt t\,\norm h_4
 \lesssim r^{1/4}+1.
\]
Thus no dimension-free estimate of the required form can follow from the coordinatewise Minkowski bound. The pointwise square function in \eqref{eq:inverse-gradient-square} is essential to the argument.

\subsection{Relation to other \texorpdfstring{$X_p$}{Xp} and cube estimates}

Naor's arbitrary-chaos estimate and metric transference appear in \cite{Naor2016}. Trigonometric and group-algebraic extensions were developed in \cite{CanoCondeParcet2024,CanoCondeParcet2025}. Dimension-free estimates for low-degree and tail-space functions on the cube are studied in \cite{DomelevoEtAl2025}. By contrast, \cite{Volberg2022,IvanisviliVolberg2022} concern Pisier- and Riesz-type inequalities for Banach targets and $K$-convexity; they are related to the operator-theoretic background but are not low-degree or tail-space results. Earlier square-function and Pisier-type developments include \cite{HytonenNaor2013,IvanisviliNazarovVolberg2018}.

The sharp $p/\log p$ dependence in \Cref{thm:main-chaos} comes from the fixed-cardinality Rosenthal comparison; the inverse-gradient estimate contributes no further dependence on $p$. To the best of the author's knowledge, the exact restriction factorization has not previously been combined with the sharp-order fixed-cardinality comparison and the inverse-gradient reverse-Poincar\'e bound to obtain the two asymptotic conclusions stated here. This statement is limited to publicly available sources and does not exclude unpublished or simultaneous work.

\subsection*{Disclosure and acknowledgments}

An OpenAI language model was used to search public literature, test finite-dimensional identities, identify possible proof gaps, and assist in preparing the manuscript.

\end{document}